\documentclass[12pt,reqno]{amsart}
\usepackage[english]{babel}
\usepackage{amsthm}
\usepackage{inputenc}
\usepackage{amssymb}
\usepackage{amsmath,amscd,amsfonts}
\newtheorem{thm}{Theorem}[section]
\newtheorem{cor}[thm]{Corollary}
\newtheorem{lm}[thm]{Lemma}
\newtheorem{pr}[thm]{Proposition}
\newtheorem{defn}[thm]{Definition}

\newtheorem{exam}[thm]{Example}
\def\Der{\operatorname{Der}}
\def\Aut{\operatorname{Aut}}
\def\Ker{\operatorname{Ker}}

\title{Some radicals, Frattini and Cartan Subalgebras of Leibniz $n$-algebras}
\author[F. Gago \and M.  Ladra \and  B.A.
Omirov \and R.M. Turdibaev]{F. Gago \and M.  Ladra \and  B.A. Omirov \and R.M. Turdibaev}

\begin{document}

\begin{abstract}
In the present work we introduce notions such as $k$-solvability, $s$-
and $K_1$-nilpotency and the corresponding radicals. We prove that
these radicals are invariant under derivations of Leibniz
$n$-algebras. The Frattini and Cartan subalgebras of Leibniz
$n$-algebras are studied. In particular, we construct examples
that show that a classical result on conjugacy of Cartan
subalgebras of Lie algebras, which also holds in Leibniz algebras and
Lie $n$-algebras, is not true for Leibniz $n$-algebras.
\end{abstract}

\keywords{Lie $n$-algebras; Leibniz $n$-algebras;
$k$-solvability, nilpotency; Frattini subalgebra; Cartan
subalgebra; Jacobson radical}
\subjclass[2010]{17A32, 17A42, 17A60.}

\maketitle

\section{Introduction}
This work is devoted to the investigation of Leibniz $n$-algebras.
In 1985, Filippov \cite{Fil} introduced a notion of Lie
 $n$-algebra with an n-ary skew-symmetric multiplication which satisfies
the identity
\[
[[x_1,x_2, \dots ,x_n],y_2,\dots ,y_n]=\sum_{i=1}^n \ [x_1,\dots ,
x_{i-1},[x_i,y_2,\dots,y_n],x_{i+1},\dots,x_n] \eqno(1)
\]
bearing in mind the general notion of $\Omega$-algebra considered
by Kurosh \cite{Kur}.

 Earlier in 1973, Nambu
\cite{Nambu} had constructed an example of 3-Lie algebra, where the multiplication for a triple of classical observables on
the three-dimensional phase space $\mathbb{R}^3$  was given by the Jacobian.
This bracket naturally generalizes the usual Poisson bracket from a
binary to a ternary operation.

In 1993, Loday \cite{Lo1,Lo2} introduced a non skew-symmetric
version of Lie algebras, the so-called Leibniz algebras. As a
generalization of Leibniz algebras and n-Lie algebras, in 2002,  Casas, Loday and
Pirashvili \cite{Cas1} defined n-Leibniz algebras  as a non
skew-symmetric version of Lie $n$-algebras. They also presented
constructions between the varieties of Leibniz algebras and Leibniz
$n$-algebras $(n \geq 3)$ which are not invertible.

In the present work, in Section \ref{prel}, we introduce the Frattini subalgebra  of a Leibniz $n$-algebra and establish properties extending  some results of
 Leibniz algebras and of Lie $n$-algebras.  Frattini theory was originally discovered in group
theory and further have been studied in Lie algebras
 in \cite{Mars,barnes,Stit}, in Lie $n$-algebras in \cite{Bai1,Will2}
  and in Leibniz algebras in \cite{barnes2,barnes3}. Here we show that
many results concerning  Frattini subalgebras and Frattini
ideals from the theory of Lie $n$-algebras remain true when we omit
the skew-symmetrical property of the $n$-ary multiplication.

In Section \ref{rmo},  we study the right multiplication operators in a Leibniz $n$-algebra.
 Filippov \cite{Fil} noted that the so-called right
multiplication operators play the same crucial role in the theory of Lie
$n$-algebras as in Lie algebras since they form a Lie algebra
with respect to the commutator. The space of the right multiplication
operators in Leibniz $n$-algebras also forms an ideal in the Lie
algebra of derivations. However, in the case $n\geq 3$, some well-known
properties of the right multiplication operators do not hold in
general; for instance, in \cite{cartan} it was given an example of a
Leibniz $n$-algebra which admits a non-degenerate  right multiplication operator. Because of that curious properties of these
operators, to obtain some results on  right
multiplication operators which are valid for Leibniz and Lie $n$-algebras we
must consider them with additional conditions.

In Section \ref{invar}, we study solvability and  nilpotency in
 Leibniz $n$-algebras and show  that the solvable and nilpotent radicals are  invariant under all derivations. Since multiplication in Leibniz $n$-algebras is not anti-symmetric
in all the variables, notions such as nilpotency and solvability may
be introduced in different ways depending on the position of the
multiplicand. The product in the definition of the corresponding
series is not necessarily an ideal and this makes some
arguments difficult to prove. Hence, we
introduce special notions, as $k$-solvability, nilpotency and
$K_1$-nilpotency of Leibniz $n$-algebras. Most of them agree
with the corresponding notions on particular cases: Lie $n$-algebras
\cite{Kasymov} and Leibniz algebras. We establish some properties
of $k$-solvable (nilpotent) ideals, as well.

Finally, in Section \ref{no_conj}, we construct examples that show the
non-conjugacy of Cartan
subalgebras for Leibniz $n$-algebras.
In \cite{cartan}, it was proved that  the null root subspace of the  right
multiplication operators with respect to a regular element is a
nilpotent subalgebra. Here we obtain that this subalgebra under
some restriction is a Cartan subalgebra. Moreover, a classical
result about conjugacy of Cartan subalgebras in Lie algebras that
was extended to the general cases - Leibniz  algebras \cite{Omirov2}
and Lie $n$-algebras \cite{Kasymov}, unfortunately does not hold in the
case of Leibniz $n$-algebras ($n\geq 2$). Starting with a particular Lie $n$-algebra,
 we construct  Leibniz $n$-algebras which factored out by the  ideal
$I$ generated by the elements
$[x_1,\dots,x_i,\dots,x_j,\dots,x_n]$, where $x_i=x_j$ for some
$1\leq i \neq j \leq n$, are  isomorphic to the given Lie $n$-algebra
under some conditions. These Leibniz $n$-algebras  have Cartan subalgebras of different
dimensions and therefore they are not conjugated (see Example \ref{cartan_examp}).

\section{Preliminaries}\label{prel}

\begin{defn}[\cite{Cas1}]
 A vector space $L$ with an $n$-ary multiplication
$[-,-,\dots,-]: L^{\otimes n} \to L$ is called a Leibniz $n$-algebra
if it satisfies the following identity
\begin{equation}\label{FI}
[[x_1,x_2, \dots ,x_n],y_2,\dots ,y_n]=\sum_{i=1}^n \ [x_1,\dots ,
x_{i-1},[x_i,y_2,\dots,y_n],x_{i+1},\dots,x_n]
\end{equation}
\end{defn}

It should be noted that if the product $[-,-,\dots,-]$ is
skew-symmetric in each pair of variables, i.e.
\[[x_1, x_2, \dots, x_i,
\dots, x_j, \dots, x_n]= - [x_1, x_2, \dots, x_j, \dots, x_i,
\dots, x_n] \, ,\] then this Leibniz $n$-algebra becomes a Lie
$n$-algebra.

Since in Leibniz $n$-algebras the $n$-ary multiplication is not
necessarily skew-symmetrical, basic notions such as ideals have to be
considered with additional conditions.

\begin{defn} A subspace $I$ of a Leibniz $n$-algebra $L$ is called an $s$-sided ideal of
$L$, if \[[\underbrace{L,\dots,L}_{s-1},I,\underbrace{L,\dots,
L}_{n-s}]\subseteq I.\] If $I$ is $s$-ideal for all $1\leq s \leq
n$, then $I$ is called an ideal.
\end{defn}

\begin{defn}
A proper subalgebra $M$ of a Leibniz $n$-algebra $L$ is called
maximal if the only subalgebra properly containing $M$ is $L$.
\end{defn}

\begin{defn}
The intersection of all maximal subalgebras of a Leibniz $n$-algebra $L$
is a subalgebra denoted by $F(L)$ and it  is called the Frattini
subalgebra.

The maximal ideal of $L$ that is contained in $F(L)$  is called the Frattini ideal and it is denoted by  $\phi(L)$.
\end{defn}

The following statements which hold for Lie $n$-algebras
\cite{Bai1} can be extended in a similar way to the case of
Leibniz $n$-algebras.

\begin{pr}
Let $L$ be a Leibniz $n$-algebra. Then the following statements
hold:
\begin{enumerate}
  \item If $B$ is a subalgebra of $L$ such that $B+F(L)=L$, then $B=L$.
  \item  If $B$ is a subalgebra of $L$ such that $B+\phi(L)=L$, then
$B=L$.
\end{enumerate}
\end{pr}

\begin{pr}\label{F(L)}
Let $L$ be a Leibniz $n$-algebra and $B$ an ideal of  $L$.
Then there exists a proper subalgebra $C$ of $L$ such that $L=B+C$
iff $B \not \subseteq F(L)$.
\end{pr}
Moreover, the assertion of Proposition \ref{F(L)} holds if we
substitute $\phi(L)$ for $F(L)$.

\begin{pr}
Let $C$ be a subalgebra of  $L$  and $B$ an ideal of $L$ such that $B \subseteq F(C) \ \big(B \subseteq \phi(C) \big) $.

Then $ B\subseteq F(L)$ \ \big($B\subseteq \phi(L)$, respectively\big).
\end{pr}

\begin{cor}
Let $L$ be a Leibniz $n$-algebra  and  $B$  a subalgebra of $L$
such that $F(B)  \ \big( \phi(B)\big )$ is an ideal of $ L$. Then
$F(B) \subseteq F(L)$ \ \big($\phi(B) \subseteq \phi(L)$,  respectively\big).
\end{cor}

\begin{pr}
Let $L$ be a Leibniz $n$-algebra  and  $B$  an ideal
of $L$. Then the following statements hold:
\begin{enumerate}
 \item $ (F(L)+B)/B \subseteq F(L/B)$, \ \big($(\phi(L)+B)/B \subseteq
\phi(L/B)$\big);
 \item If  $ B \subseteq F(L)$ then $F(L)/B=F(L/B), \ \ \phi(L)/B=\phi(L/B)$;
\item If $ F(L/B)=0 \ \ (\phi(L/B)=0)$, then $ F(L) \subseteq B \ \
(\phi(L)) \subseteq B$.
\end{enumerate}
\end{pr}

\begin{thm}
If a Leibniz $n$-algebra  $L$ has a decomposition \[ L=L_1\oplus
L_2\oplus \dots\oplus L_m,\] where  $L_i  \ (1 \leq i \leq m)$  are
ideals of $L$, then
\begin{enumerate}
\item $ F(L) \subseteq F(L_1) + \dots + F(L_m)$;
\item $\phi(L) = \phi(L_1)+ \dots + \phi(L_m)$.
\end{enumerate}
\end{thm}

Given an arbitrary Leibniz $n$-algebra $L$ consider the following
sequences ($s$ is a fixed natural number, $1\leq s\leq n$):
\begin{align*} L^{<1>_s}=L,\qquad  & L^{<k+1>_s}=
\displaystyle[\underbrace{L,\ldots,L}_{(s-1)-\textrm{times}},L^{<k>_s},
\underbrace{L,\ldots,L}_{(n-s)-\textrm{times}}], \\
L^1=L, \qquad &  L^{k+1}=\sum\limits_{i=1}^n \ \displaystyle[\underbrace{L,\ldots,L}_{(i-1)-
\textrm{times}},L^k,
\underbrace{L,\ldots,L}_{(n-i)-\textrm{times}}].
\end{align*}

\begin{defn} A Leibniz $n$-algebra $L$ is said to be $s$-nilpotent
(nilpotent)  if there exists a natural number $k\in
\mathbb{N}$ ($l\in \mathbb{N}$) such that
$L^{<k>_s}=0$ ($L^l=0$, respectively).
\end{defn}

It should be noticed that for Lie $n$-algebras the above notions of
$s$-nilpotency and nilpotency coincide. Recall also that for
Leibniz algebras (i.e. Leibniz $2$-algebras) the notions of
$1$-nilpotency and nilpotency also coincide \cite{Ayup}.

In \cite[Example 2.2]{cartan}, it is shown that the $s$-nilpotency
property for Leibniz $n$-algebra ($n\geq 3$) essentially depends
on $s$.

Let $H$ be an ideal of a Leibniz $n$-algebra $L$. Put
$H^{(1)_k}=H$ and
\[H^{(m+1)_k}
=\sum_{i_1+\dots+i_k=0}^{n-k}
[\underbrace{L,\dots,L}_{i_1},H^{(m)_k},\underbrace{L,\dots,L}_{i_2},H^{(m)_k},\dots,
\underbrace{L,\dots,L}_{i_k}H^{(m)_k},\underbrace{L,\dots,L}_{n-i_1-\dots-i_k}]\]
for all $1\leq k \leq n$ and $m\geq 1$.

\begin{defn} An $n$-sided ideal $H$ of Leibniz $n$-algebra is said
to be $k$-solvable with index of $k$-solvability equal to $m$ if
there exists $ m\in \mathbb{N}$ such that $H^{(m)_k}=0$ and
$H^{(m-1)_k}\neq 0$.

When $L = H$, $L$ is called  a $k$-solvable Leibniz $n$-algebra.
\end{defn}

Notice that this definition agrees with the definition of
$k$-solvability of Lie $n$-algebras given in \cite{Kasymov}.

\begin{defn} We say that a subalgebra $U$ of a Leibniz $n$-algebra $L$
is left subnormal if there exists a chain of subalgebras
$U=U_k\subseteq\cdots\subseteq U_1\subseteq U_0=L$ with each
$U_{i+1}$ an $r$-ideal $(r\neq 1)$ in $U_i$.
\end{defn}

\begin{thm} Let $U$ be a left subnormal subalgebra of Leibniz
$n$-algebra $L$ and $V$ an ideal in $U$ such that $V\subseteq
F(L)$. If $U/V$ is $1$-nilpotent, then $U$ is $1$-nilpotent.
\end{thm}

\begin{proof}
Similar to the proof of  \cite[Theorem 3.6]{barnes2}.
\end{proof}

The following statements hold for Lie $n$-algebras
\cite{Bai1} and are also true for Leibniz $n$-algebras.

\begin{cor} If $I \subseteq F(L)$ is an $r$-ideal $(r\neq
1)$ of $L$, then $I$ is $1$-nilpotent. Particularly, $\phi(L)$ is a
$1$-nilpotent ideal of $L$.
\end{cor}

\begin{defn}
In a Leibniz $n$-algebra $L$ the intersection of all maximal
 ideals of $L$ is called the Jacobson radical and it is denoted by
$J(L)$.
\end{defn}

\begin{pr}\label{onesidewilliams}
Let $L$ be a finite dimensional Leibniz $n$-algebra. Then
\[F(L)\subseteq[L,L,\dots,L]\textrm{ and } J(L)\subseteq [L,L,\dots,L].\]
Moreover, if $L$ is a $k$-solvable Leibniz $n$-algebra, then
\[J(L)=[L,L,\dots,L].\]
\end{pr}

\begin{thm} Let $L$ be a finite dimensional nilpotent Leibniz $n$-algebra. Then the following statements hold:
\begin{enumerate}
  \item Any maximal subalgebra $M$ of $L$ is an ideal of $L$;
  \item  $F(L)=\phi(L)=J(L)=[L,L,\dots,L]$.
\end{enumerate}
\end{thm}

\section{Right multiplication Operators}\label{rmo}

\begin{defn} A linear map $d$ defined on a Leibniz $n$-algebra $L$ is called a derivation  if
\[d([x_1,x_2,\dots , x_n])=\sum_{i=1}^n \ [x_1,\dots d(x_i),\dots , x_n].\]
The space of all derivations of a given Leibniz $n$-algebra $L$ is
denoted by $\Der(L)$.
\end{defn}

The space $\Der(L)$ forms a Lie algebra with respect to the
commutator \cite{cartan}.

Set $A^{\times k}=\underbrace{A\times A \times \dots \times
A}_{k-times}$.

Given an arbitrary element $x=(x_2,\dots,x_n) \in L^{\times(n-1)}$
consider the operator $R(x): L \to L$ of right multiplication
defined by \[R(x)(z)=[z,x_2,\dots,x_n].\]

Any right multiplication operator is a derivation and the space
$R(L)$ of all right multiplication operators forms a Lie ideal of
$\Der(L)$ \cite{cartan}.

\begin{thm}[\cite{cartan} Engel's theorem] A Leibniz $n$-algebra
$L$ is 1-nilpotent if and only if $R(x)$ is nilpotent for all
$x\in L^{\times(n-1)}$.
\end{thm}

In \cite{cartan} it was given an example of a Leibniz $n$-algebra
which admits a non-degenerated right multiplication operator.
This is the significant difference between Leibniz $n$-algebras ($n\geq
3$) on one hand and  Leibniz algebras and Lie $n$-algebras on the other.

Below we assume that all right multiplication operators are
degenerated.

The following lemma yields a decomposition of a given vector space
into a direct sum of two subspaces which are invariant with
respect to a given linear transformation.

\begin{lm}[Fitting Lemma] Let $V$ be a vector space and
$A:V\to V$ be a linear transformation. Then $V= V_{0A}\oplus
V_{1A}$, where $A(V_{0A})\subseteq V_{0A}$, $A(V_{1A})\subseteq
V_{1A}$ and $V_{0A}=\{ v \in V|\ A^i(v)=0 \ \mbox{for some} \ i \}$ and
$V_{1A}=\bigcap\limits_{i=1}^\infty A^i(V)$. Moreover, $A_{|
V_{0A}}$ is a nilpotent transformation and $A_{|V_{1A}}$ is an
automorphism. $V_{0A}$ is called the Fitting null-component of $V$ with
respect to $A$.
\end{lm}
\begin{proof}
See \cite[Chapter II, \S 4]{Jac}.
\end{proof}


\begin{defn} An element $h \in L^{\times (n-1)}$ is said to be regular
for the algebra $L$ if the dimension of the Fitting null-component of the space $L$ with respect to $R(h)$ is minimal.
\end{defn}

\begin{lm}[\cite{onnilpotent}]\label{decomposition} Let $L$ be a finite dimensional complex Leibniz n-algebra with
given derivation $d$, and let $L=L_{\alpha}\oplus
L_{\beta}\oplus\cdots\oplus L_{\gamma}$ be the decomposition of
the algebra $L$ into root spaces with respect to the derivation
$d$ (i.e. $L_{\alpha}=\{x\in L | \ (d-\alpha I)^k x=0  \  \mbox{for
some} \  k \}$). Then
\[[L_{\alpha_1},L_{\alpha_2},\dots,L_{\alpha_n}]\subseteq\left\{\begin{array}{ll}
0 & \mbox{if } \alpha_1+\alpha_2+\cdots+\alpha_n \mbox{ is not a root of d}\\
L_{\alpha_1+\alpha_2+\cdots+\alpha_n} & \mbox{if }
\alpha_1+\alpha_2+\cdots+\alpha_n \mbox{ is a root of d} \, .
\end{array}\right.\]
\end{lm}

\begin{pr} \label{zerosum}In a Leibniz $n$-algebra $L$ any right multiplication operator
 $ R(a_2,\dots,a_n)$ is a sum of right multiplication operators with zero
 root space with respect to $ R(a_2,\dots,a_n)$.
\end{pr}

\begin{proof}
Let $ \alpha_0=0, \alpha_1, \dots, \alpha_k$ be the eigenvalues of
$ R(a_2,\dots,a_n)$. Then $L$ is decomposed into a direct sum
\[L=L_0 \oplus L_{\alpha_1} \oplus \dots \oplus L_{\alpha_k},\] where $
L_{\alpha_i}=\{x\, | \,(R(a_2,\dots,a_n)-\alpha_i I)^m(x)=0 \mbox{
for some } m \in \mathbb{N}\}$.

Consider $ a_i=a_{0}^i+a_{\alpha_1}^i+\dots+a_{\alpha_k}^i,
\,a_{\alpha_m}^i \in L_m, \,  2 \leq i \leq k$. Then for all $ x
\in L$, we have
\begin{multline*}
R(a_2,\dots,a_n)(x)=[x,a_2,\dots , a_n]=
[x,a_{0}^2+a_{\alpha_1}^2+\dots+a_{\alpha_k}^2,\dots ,
a_{0}^n+a_{\alpha_1}^n+\dots+a_{\alpha_k}^n] \\
=[x,a_{0}^2,\dots,a_{0}^n]+[x,a_{\alpha_1}^2,a_0^3\dots,a_{0}^n]+\dots
+[x,a_{\alpha_k}^2,a_{\alpha_k}^3, \dots,a_{\alpha_k}^n]\\
=R(a_{0}^2,a_{0}^3,\dots,a_{0}^n)(x)+R(a_{\alpha_1}^2,a_{0}^3\dots,a_{0}^n)(x)+
\dots + R(a_{\alpha_k}^2,a_{\alpha_k}^3,
\dots,a_{\alpha_k}^n)(x) \, .
\end{multline*}

By Lemma \ref{decomposition}, we obtain that $R(a_2,\dots,a_n)(x)=B(x)
+C(x)$, where $B$ is a sum of right multiplication operators
with zero weight and $C$ is a sum of right multiplication operators with nonzero weights. Then for any $ x\in
L_{\alpha_i}$, we have
\[C(x)=(R(a_2,\dots,a_n)-B)(x) =R(a_2,\dots,a_n)(x)-B(x) \subseteq
L_{\alpha_i} \, ,\] which holds only if  $ C(x)=0$ since  $C$ adds a
weight. Therefore, $C$ is a zero operator on $L_{\alpha_i}$. Since
$ L=L_0 \oplus L_{\alpha_1}\oplus \dots \oplus L_{\alpha_k}$, we
obtain $C=0$ on $L$.

So, $R(a_2,\dots,a_n)=B$, i.e. is a sum of right multiplication operators  with zero weight with respect to
$R(a_2,\dots,a_n)$.
\end{proof}

In \cite{barnes2} it was proved the following result for left
Leibniz algebras which is also valid for right Leibniz algebras,
i.e. Leibniz 2-algebras.

\begin{lm}[\cite{barnes2}] In a Leibniz algebra $L$ for any $a\in L$ there exists
$b\in L_0(R_a)$ such that $L_0(R_b)=L_0(R_a)$.
\end{lm}

Concerning this lemma we establish the following result for the
case $n\geq 3$.

\begin{cor}\label{eigen} If the nonzero eigenvalues $\alpha_1, \dots ,
\alpha_k$ of the right multiplication operator $R(a_2,\dots, a_n)$ in a
Leibniz $n$-algebra $(n\geq 3)$ satisfy
\[\mu_1\alpha_1+\mu_2 \alpha_2+\dots +\mu_k \alpha_k\neq
0,\] for all non-negative integers $ \mu_1,\dots, \mu_k$ such that
\[0<\mu_1+\dots+\mu_k \leq n-1,\]
then there exist $ b_2,b_3,\dots, b_n \in L_0(R(a_2,\dots, a_n))$
such that
\[L_0(R(b_2,\dots,
b_n))=L_0(R(a_2,\dots, a_n)).\]
\end{cor}

\begin{proof}
From Proposition \ref{zerosum} we obtain that $ R(a_2,\dots, a_n) = B$.
From the condition on the eigenvalues we conclude that $B$ consists of
just one right multiplication operator, namely
$B=R(a_{0}^2,a_{0}^3,\dots,a_{0}^n)$. So, if we take  $b_i= a_0^i$,
we obtain $ L_0(R(b_2,\dots,b_n))=L_0(R(a_2,\dots, a_n))$.
\end{proof}

A Leibniz $n$-algebra satisfying the conditions of Corollary
\ref{eigen} is given in the following
\begin{exam}[\cite{cartan}] Consider a Leibniz $n$-algebra $L=\langle e_1,e_2,\dots,e_n\rangle$
with the following multiplication:
\[[e_k,e_1,\dots,e_1]=e_k \,\,\,\,\, (2 \leq k \leq m).\]

The right multiplication operator  $R(e_1,\dots,e_1)$ has only two
eigenvalues: $0$ and $1$. It is easy to see that the conditions of
Corollary \ref{eigen} are satisfied and $e_1\in
L_0(R(e_1,\dots,e_1))$.
\end{exam}

Below, we present an example which shows the sufficiency of the
condition in Corollary \ref{eigen}.

\begin{exam} Consider an $m$ dimensional Leibniz $n$-algebra $L$ with
the following multiplication:
\[\begin{array}{rcl}
  [e_k,e_1,e_2\dots,e_{n-1}] & = & \alpha_k e_k \\
 \, [e_{k+1},e_1,e_2,\dots, e_{n-1}] & = & \alpha_{k+1} e_{k+1} \\
& \vdots &  \\
 \, [e_m,e_1,e_2, \dots, e_{n-1}] & = & \alpha_m e_m \\
\end{array}\]

where $ \{e_1,\dots, e_m\}$ is a basis, $ k < n-1$ and
 \ $\displaystyle \sum_{i=k}^{n-1}\alpha_i=0, \alpha_k\cdots \alpha_m \neq 0$.

Then $ L_0(R(e_1,\dots, e_{n-1}))=\{e_1,\dots, e_{k-1}\}$. Since
any other right multiplication operator either coincides with
$R(e_1,\dots, e_{n-1})$ or is identically zero, there does not
exist $b_2,\dots , b_n \in L_0(R(e_1,\dots, e_{n-1}))$ such that $
L_0(R(b_2, \dots, b_n))=L_0(R(e_1,\dots, e_{n-1}))$.
\end{exam}

\begin{defn}[\cite{cartan}] Given a subset $X$ in a Leibniz $n$-algebra, the
$s$-normalizer of $X$ is the set \[N_s(X)=\{a\in L \,|\,
[x_1,\dots,x_{s-1},a,x_{s+1},\dots,x_n]\in X \textrm{ for all }
x_i\in X\}.\] The set $\displaystyle N(X)=\bigcap_{s=1}^n N_s(X)$
is called the normalizer of $X$.
\end{defn}
Notice that, if $X$ is a subalgebra of $L$, then $N(X),
N_s(X)\supseteq X$.

\begin{lm}[\cite{cartan}] \label{desc_lem} Let $M$ be an invariant subspace of a vector space $L$
with respect to a linear transformation $Q:L\to L$. Let
 $x=x_0+x_{\alpha}+x_{\beta}+\dots+x_{\gamma}$ be any  decomposition of an element $x$
into a sum of characteristic vectors from the corresponding
characteristic spaces $L_\xi (\xi \in \{0,\alpha,\beta,\dots,
\gamma\})$. If $Q(x)\in M$, then $x-x_0 \in M$.
\end{lm}

The following lemma is an extension of
\cite[Lemma 3.2]{barnes2} under the condition $a_2,\dots, a_n \in
L_0(R(a_2,\dots,a_n))$.

\begin{lm} Let $L$ be a Leibniz $n$-algebra and
$R(a_2,\dots,a_n): L \to L$  a right multiplication operator  such that $a_2,\dots, a_n \in
L_0(R(a_2,\dots,a_n))$. Then for any subalgebra $U$ containing
$L_0(R(a_2,\dots,a_n))$ the equality $ N(U)=U$ holds.
\end{lm}

\begin{proof}
Let $z\in N_1(U)$. Then $ [z,U,\dots, U] \subseteq U$. Denote
$L_0= L_0(R(a_2,\dots,a_n))$. Then
\[R(a_2,\dots,a_n)(z)=[z,a_2,\dots,a_n]\in [z,L_0,\dots, L_0]
\subseteq [z,U,\dots, U] \subseteq U.\] Hence $ R(a_2,\dots,
a_n)(z) \in U$.

Notice that $ R(a_2,\dots, a_n) (U) =[U,a_2,\dots, a_n]\subseteq
[U,L_0,\dots, L_0]\subseteq [U,U,\dots, U] \subseteq U$ since $U$
is a subalgebra. Therefore, the conditions of Lemma \ref{desc_lem} are
satisfied. Thus $ z-z_0 \in U$. Then $ z\in U$. So we have proved
$ N_1(U)=U$.

Since $U$ is a subalgebra, $N_s(U)\supseteq U$ for all $2\leq s
\leq n$. Then $N(U)=N_1 \bigcap \big(\cap_{s=2}^n
N_s(U)\big)=U$.
\end{proof}

\begin{pr} Let $a_2,\dots, a_n $ be elements  of a Leibniz $n$-algebra $L$  such that
$a_2,\dots, a_n \in L_0(R(a_2,\dots,a_n))$. If every maximal
subalgebra is an $i$- and  a $j$-ideal $(1\leq i\neq j\leq n)$ in $L$,
then $R(a_2,\dots,a_n)$ is nilpotent.
\end{pr}

\begin{proof} Assume that $L_0(R(a_2,\dots,a_n)) \neq L$.
Then there exists maximal algebra $M$ such that
$L_0(R(a_2,\dots,a_n)) \subseteq M$. Then by previous lemma we
have $N(M)=M$.

Since $M$ is an $i$- and a $j$-ideal $(i\neq j)$, we have
\[[\underbrace{M,\dots M}_{s-1},L,M,\dots, M] \subseteq M\] for
all $1\leq s \leq n$. Thus, $ L=N_s(M)$ for all $1\leq s \leq n$
and $L=N(M)$. Contradiction.

Therefore  $L=L_0(R(a_2,\dots,a_n))$ and $R(a_2,\dots,a_n)$ is a
nilpotent operator.
\end{proof}

In \cite[Theorem 2.2]{Will1} there were given several statements
equivalent to nilpotency of the finite dimensional Lie
$n$-algebras. For Leibniz $n$-algebras Proposition
\ref{onesidewilliams} verifies the statement in one direction. The
other direction of the statement in our
case is not true in general. However we establish the following result.

\begin{pr}Let $L$ be a finite dimensional Leibniz $n$-algebra with
condition $a_i\in L_0(R(a_2,\dots,a_n))$ for some $2\leq i \leq n$
for an arbitrary $(a_2,\dots, a_n)\in L^{\times (n-1)}$. If any
maximal subalgebra $M$ of $L$ is an ideal of $L$ then
$L$ is $1$-nilpotent.
\end{pr}
\begin{proof} Assume that $L$ is not $1$-nilpotent. Then there
exists a non-nilpotent right multiplication operator
$R(a_2,\dots,a_n)$. Since $R(a_2,\dots,a_n)$ is non-nilpotent,
the Fitting null-component $L_0(R(a_2,\dots,a_n))\neq L$.

Let $M$ be a maximal subalgebra of $L$ containing
$L_0(R(a_2,\dots,a_n))$. Then $a_i\in
L_0(R(a_2,\dots,a_n))\subseteq M$ for some $2\leq i \leq n$ by
assumption of the proposition. Since $M$ is a maximal subalgebra, it
is also an ideal of $L$. Then
$R(a_2,\dots,a_n)(L)\subseteq M$.

Since $R(a_2,\dots,a_n)$ is an automorphism on
$L_1(R(a_2,\dots,a_n))$, we obtain that $L_1=R(L_1)=L_1\cap M$. Hence
$L_1 \subseteq M$.

Then $L=L_0\oplus L_1\subseteq M \neq L$. This is a contradiction.
Hence, all right multiplication operators are nilpotent.
Therefore, by Engel's theorem $L$ is $1$-nilpotent.
\end{proof}

\section{Invariance of some radicals under derivation}\label{invar}

In the following section we establish some classical results from the
theory of Lie algebras concerning solvability and nilpotency which
are also true in Leibniz algebras and Lie $n$-algebras.

\begin{pr}\label{equality} For an ideal $H$ of a Leibniz $n$-algebra $L$
the equality $(H^{(m)_k})^{(r)_k} =H^{(m+r-1)_k}$ holds for all
$m,r \in \mathbb{N}$.
\end{pr}

\begin{proof} Using induction on $r$ one can easily prove the
assertion of the proposition.
\end{proof}

Even though we can not state that $H^{m_k}$ is an $s$-sided ideal
for all $1\leq s \leq n$, we establish the following result.

\begin{pr} For an ideal $H$ of a Leibniz $n$-algebra $L$,
$H^{(m)_k}$ is a $1$-ideal of $L$ for all $m,k \in \mathbb{N}$.
\end{pr}

\begin{proof} Let $k$ be an arbitrary fixed natural number.

For $m=1$ we have $[H^{(1)_k},L,\dots,L] \subseteq [H,L,\dots,L]
\subseteq H=H^{(1)_k}$ since $H$ is an ideal.

Let $H^{(m)_k}$ be a $1$-ideal, i.e. $[H^{(m)_k},L,\dots,L]\subseteq
H^{(m)_k}$.

Then
\begin{multline*}
[H^{(m+1)_k},L,\dots,L] \\
 =\left[\sum_{i_1+\dots+i_k=0}^{n-k} \
[\underbrace{L,\dots,L}_{i_1},H^{(m)_k},\dots,
\underbrace{L,\dots,L}_{i_k},H^{(m)_k},\underbrace{L,\dots,L}_{n-i_1-\dots-i_k}],L,\dots,L\right] \\
=\sum_{i_1+\dots+i_k=0}^{n-k}
\Big[[\underbrace{L,\dots,L}_{i_1},H^{(m)_k},\dots,
\underbrace{L,\dots,L}_{i_k},H^{(m)_k},\underbrace{L,\dots,L}_{n-i_1-\dots-i_k}],L,\dots,L\Big] \, .
\end{multline*}

Since $H^{(m)_k}$ is a $1$-ideal by induction hypothesis, using
identity \eqref{FI} we obtain that
\begin{multline*}
\Big[[\underbrace{L,\dots,L}_{i_1},H^{(m)_k},\dots,
\underbrace{L,\dots,L}_{i_k},H^{(m)_k},\underbrace{L,\dots,L}_{n-i_1-\dots-i_k}],L,\dots,L\Big]\\
 \subseteq [\underbrace{L,\dots,L}_{i_1},H^{(m)_k},\dots,
\underbrace{L,\dots,L}_{i_k},H^{(m)_k},\underbrace{L,\dots,L}_{n-i_1-\dots-i_k}] \, .
\end{multline*}
Therefore
\begin{multline*}
[H^{(m+1)_k},L,\dots,L] \\
\subseteq \sum_{i_1+\dots+i_k=0}^{n-k}
[\underbrace{L,\dots,L}_{i_1},H^{(m)_k},\dots,
\underbrace{L,\dots,L}_{i_k},H^{(m)_k},\underbrace{L,\dots,L}_{n-i_1-\dots-i_k}]=H^{(m+1)_k}
\end{multline*}
and $H^{(m+1)_k}$ is a $1$-ideal of $L$.
\end{proof}

\begin{pr}\label{k-sol} Let $I$ be a $k$-solvable ideal of a Leibniz $n$-algebra
$L$ such that $L/I$ is also $k$-solvable. Then $L$ is
$k$-solvable.
\end{pr}

\begin{proof}Let $\phi:L \to L/I$ be the natural homomorphism. Since $L/I$ is $k$-solvable, we have
$0=( L/I)^{(m)_k}=\big(\phi(L)\big)^{(m)_k}=\phi(L^{(m)_k})$ for some
$m\in \mathbb{N}$. Thus $L^{(m)_k} \subseteq I$. Since $I$ is
$k$-solvable, there exists $p\in \mathbb{N}$ such that
$I^{(p)_k}=0$. Therefore by Proposition \ref{equality} we have
$L^{(m+p-1)_k}=(L^{(m)_k})^{(p)_k}\subseteq I^{(p)_k}=0$ and so $L$
is $k$-solvable.
\end{proof}
By induction it is easy to prove that if $I$ is a $k$-solvable
ideal of a Leibniz $n$-algebra $L$, then $I$ is also
$(k+p)$-solvable for all $p\in \mathbb{N}$.

Using standard methods  and Proposition \ref{k-sol} we obtain that
the sum of $k$-solvable ideals is also $k$-solvable. Now
let $H$ be a maximal $k$-solvable  ideal in a finite
dimensional Leibniz $n$-algebra $L$ and let $K$ be an arbitrary
$k$-solvable  ideal of $L$. Then $H+K$ is also
$k$-solvable and $H+K \supseteq H$. Since $H$ is a maximal
$k$-solvable ideal, we obtain that $H+K=H$. Therefore we can
define the maximal $k$-solvable ideal as the  sum of all the
$k$-solvable ideals in $L$ and call it the \emph{$k$-solvable radical}.

 The following formula for a
derivation $d:L \to L$ of a Leibniz $n$-algebra $L$ over a field
$\mathbb{K}$ of characteristic zero, for any $k \in \mathbb{N}$, was given in \cite{onnilpotent}:
\begin{equation} \label{der}
d^k([x_1,\dots,x_n]) = \sum_{i_1+i_2+\cdots+i_n=k} \ \frac{k!}{i_1!i_2!\dots
i_n!} \ [d^{i_1}(x_1),d^{i_2}(x_2),\dots,d^{i_n}(x_n)] \, .
\end{equation}

\begin{pr} \label{inclusion} Let $I$ be an ideal of a
Leibniz $n$-algebra $L$ and $d \in \Der(L)$. Then \[\big(d(I)\big)^{(m)_k}
\subseteq I+d^{k^{m-1}}\big(I^{(m)_k}\big)\] for all $m\in \mathbb{N}$ and
$1\leq k \leq n$.
\end{pr}
\begin{proof} For $m=1$ we have $d(I)\subseteq
I+d(I)$ which obviously holds.

Assume that $\big(d(I)\big)^{(m)_k} \subseteq I+d^{k^{m-1}}\big(I^{(m)_k}\big)$.

Using formula \eqref{der} we verify the inclusion for $m+1:$

\begin{multline*}\big(d(I)\big)^{(m+1)_k}\\
= \sum_{i_1+\dots+i_k=0}^{n-k}
[\underbrace{L,\dots,L}_{i_1},d(I)^{(m)_k},\underbrace{L,\dots,L}_{i_2},d(I)^{(m)_k},\dots,
\underbrace{L,\dots,L}_{i_k},d(I)^{(m)_k},\underbrace{L,\dots,L}_{n-i_1-\dots-i_k}]\\
\subseteq \sum_{i_1+\dots+i_k=0}^{n-k}
[\underbrace{L,\dots,L}_{i_1},I+d^{k^{m-1}}(I^{(m)_k}),\dots,
\underbrace{L,\dots,L}_{i_k},I+d^{k^{m-1}}(I^{(m)_k}),\underbrace{L,\dots,L}_{n-i_1-\dots-i_k}]\\
\subseteq I+d^{k^m}\left(\sum_{i_1+\dots+i_k=0}^{n-k}
[\underbrace{L,\dots,L}_{i_1},I^{(m)_k},\dots,
\underbrace{L,\dots,L}_{i_k},I^{(m)_k},\underbrace{L,\dots,L}_{n-i_1-\dots-i_k}]\right)
=I+d^{k^m}\big(I^{(m+1)_k}\big) \, .
\end{multline*}

Therefore, the assertion of the proposition is true.
\end{proof}

Also, in \cite{onnilpotent}, it was shown that for any ideal I
of $L$ and $d\in \Der(L)$ the $I+d(I)$ is also an ideal
of $L$.

\begin{thm}\label{sol_rad} Let $J$ be the $k$-solvable radical of a finite
dimensional Leibniz $n$-algebra $L$ over a field $\mathbb{K}$
 of characteristic zero. Then $d(J)\subseteq J$ for any $d\in
\Der(L)$.
\end{thm}

\begin{proof}
Let $s\in \mathbb{N}$ be such $J^{(s)_k}=0$. Then by Proposition
\ref{inclusion} we have $\big(d(J)\big)^{(s)_k} \subseteq
J+d^{k^{s-1}}\big(J^{(s)_k}\big)=J$.

Using formula \eqref{der}, we obtain that $\big(J+d(J)\big)^{(s)_k}\subseteq
J+\big(d(J)\big)^{(s)_k} \subseteq J$. Now by Proposition \ref{equality} we
have $\big(J+d(J)\big)^{(2s-1)_k}=\Big(\big(J+d(J)\big)^{(s)_k}\Big)^{(s)_k}\subseteq
J^{(s)_k}=0$. But this means that $J+d(J)$ is a $k$-solvable ideal.
Since $J$ is a $k$-solvable radical, we obtain that $J+d(J)\subseteq J$
and therefore $d(J) \subseteq J$.
\end{proof}

Similarly as in \cite{Bai2} we introduce the following series
for a $1$-sided ideal $I$ of a Leibniz $n$-algebra $L:$

\[I^{[1]}=I, \quad  I^{[k+1]}=[I^{[k]},I,L,\dots,L] \quad (k\geq 1).\]

By a simple induction using identity \eqref{FI} it can be proved that for
any $1$-sided ideal $I$ and for all $n\in\mathbb{N}$, $I^{[n]}$ is a
$1$-sided ideal.

\begin{defn} A $1$-sided ideal $I$ is called $K_1$-nilpotent, if
there exists $k\in\mathbb{N}$ such that $I^{[k]}=0$.
\end{defn}

The introduced type of nilpotency is also known as nilpotency in
the sense of Kuzmin for Lie $n$-algebras. Identity \eqref{FI} is
organized in such way, that the elements of the above introduced
series are $1$-ideals. However, if we change the position of
$I^{[k]}$ in the product defining $I^{[k+1]}$ from the first to
any other, we are not able to state that the elements of the obtained
series will be $s$-ideals of $L$ for any $2\leq s \leq n$.

\begin{pr} Let $I$ and $J$ be $K_1$-nilpotent $1$-sided ideals. Then
$I+J$ is also a $K_1$-nilpotent $1$-sided ideal.
\end{pr}
\begin{proof} First, observe that
\[[I^{[p]}\cap J^{[q]},I,L\dots,L]\subseteq [I^{[p]},I,L,\dots,L]=I^{[p+1]} \, ,\]
and since $J^{[q]}$ is a $1$-ideal, we get
\[ [I^{[p]}\cap J^{[q]},I,L\dots,L]\subseteq
[J^{[q]},I,L\dots,L]\subseteq J^{[q]}.\]
Therefore,
\[[I^{[p]}\cap J^{[q]},I,L\dots,L]\subseteq I^{[p+1]}\cap
J^{[q]}.\] Analogously,
\[[I^{[p]}\cap J^{[q]},J,L\dots,L]\subseteq I^{[p]}\cap
J^{[q+1]}.\]

We have $(I+J)^{[1]}=I+J=I^{[1]}+J^{[1]}$.

Now assume that
\[(I+J)^{[k]}\subseteq I^{[k]}+ \big(I^{[k-1]}\cap J^{[1]}\big) +\dots +
 \big( I^{[1]}\cap J^{[k-1]} \big)+J^{[k]}.\]

Then
\begin{multline*}
(I+J)^{[k+1]}=[(I+J)^{[k]},I+J,L,\dots,L]\\
 \subseteq [(I+J)^{[k]},I,L,\dots,L]+[(I+J)^{[k]},J,L,\dots,L]\\
\subseteq [I^{[k]},I,L,\dots,L]+\sum_{r=1}^{k-1} \  [I^{[k-r]}\cap
J^{[r]},I,L,\dots,L]+[J^{[k]},I,L,\dots,L]\\
+ [I^{[k]},J,L,\dots,L]+\sum_{r=1}^{k-1} \  [I^{[k-r]}\cap
J^{[r]},J,L,\dots,L]+[J^{[k]},J,L,\dots,L]\\
 \subseteq I^{[k+1]}+\Big(\sum_{r=1}^{k-1} \ I^{[k-r+1]}\cap
J^{[r]}\Big)+\big(I^{[1]}\cap J^{[k]}\big) \\
+\big(I^{[k]}\cap J^{[1]}\big)+\Big(\sum_{r=1}^{k-1} \
I^{[k-r]}\cap J^{[r+1]}\Big)+J^{[k+1]}\\
\subseteq I^{[k+1]}+\big( I^{[k]}\cap J^{[1]}\big)+\dots +\big(I^{[1]}\cap J^{[k]}\big)+J^{[k+1]} \,.
\end{multline*}

Hence, for any $n\in\mathbb{N}$ we have \[(I+J)^{[n]}\subseteq
I^{[n]}+\big(I^{[n-1]}\cap J^{[1]}\big)+\dots +\big(I^{[1]}\cap J^{[n-1]}\big)+
J^{[n]}.\]

So if $I^{[n_1]}=0$ and $J^{[n_2]}=0$, then for $n=n_1+n_2$ every
summand in the above sum is zero. Therefore $(I+J)$ is also
$K_1$-nilpotent.
\end{proof}

\begin{cor} Let $I$ and $J$ be $K_1$-nilpotent ideals. Then
$I+J$ is also a $K_1$-nilpotent ideal.
\end{cor}

Let $I$ be a maximal $K_1$-nilpotent ideal in a finite
dimensional Leibniz $n$-algebra $L$ and let $J$ be an arbitrary
$K_1$-nilpotent  ideal of $L$. Then $I+J$ is also
$K_1$-nilpotent and $I+J \supseteq I$. Since $I$ is a maximal
$K_1$-nilpotent  ideal, we obtain that $I+J=I$. Therefore we
can define the maximal $K_1$-nilpotent ideal as the sum of
all the $K_1$-nilpotent ideals in $L$ and call it the \emph{$K_1$-nilradical}. Notice that, the $K_1$-nilradical do not possess the
properties of the radical in the sense of Kurosh.
\\

Using the same argumentation as in the proof of Proposition
\ref{inclusion} and Theorem \ref{sol_rad} the following statements
can be established.
\begin{pr} Let $I$ be an ideal of a Leibniz
$n$-algebra $L$. Then for any $d\in \Der(L)$ we have
$\big(d(I)\big)^{[n]}\subseteq I+d^n(I^{[n]})$ for all $n\in\mathbb{N}$.
\end{pr}

\begin{thm} Let $J$ be the $K_1$-nilradical of a Leibniz
$n$-algebra $L$. Then for any $d\in \Der(L)$ we have $d(J)\subseteq
J$.
\end{thm}

Analogously, we can establish similar results concerning the
nilpotency and $s$-nilpotency.

By induction it is not difficult to show that the sum of
$s$-nilpotent (nilpotent) ideals of Leibniz $n$-algebra
$L$ is also $s$-nilpotent (nilpotent, respectively)
ideal of $L$.

Now let $N$ be a maximal $s$-nilpotent (nilpotent)ideal
in a finite dimensional Leibniz $n$-algebra $L$ and let $M$ be an
arbitrary $s$-nilpotent ideal of $L$. Then  $N+M$ is
also $s$-nilpotent (nilpotent, respectively) and $N+M \supseteq
N$. Since $N$ is maximal $s$-nilpotent (nilpotent, respectively)
ideal, we obtain $N+M=N$. Therefore we can define the
maximal $s$-nilpotent (nilpotent, respectively) ideal as
the sum of all the  $s$-nilpotent (nilpotent, respectively) ideals in $L$
and call it the \emph{$s$-nilradical} (\emph{nilradical}, respectively).

\begin{pr}Let $J$ be the $s$-nilradical (nilradical) of a finite
dimensional Leibniz $n$-algebra $L$ over a field $\mathbb{K}$
 of characteristic zero. Then $\big(J+d(J)\big)^{<m>_s} \subseteq
J^{<m>_s}+\big(d(J)\big)^{<m>_s}$  \ \Big($\big(J+d(J)\big)^{m} \subseteq
J^{m}+\big(d(J)\big)^{m}$, respectively\Big).
\end{pr}

\begin{proof} Analogous to the proof of Proposition \ref{inclusion}.
\end{proof}

\begin{thm}Let $J$ be the $s$-nilradical (nilradical) of a finite
dimensional Leibniz $n$-algebra $L$ over a field $\mathbb{K}$ of
characteristic zero. Then $d(J)\subseteq J$ for any $d\in
\Der(L)$.
\end{thm}

\begin{proof} Analogous to the proof of Theorem \ref{sol_rad}.
\end{proof}

\section{Non-conjugacy of Cartan Subalgebras}\label{no_conj}

In this section we consider Cartan and Frattini subalgebras of
Leibniz $n$-algebras.

\begin{defn}[\cite{cartan}]  A subalgebra $C$ of a Leibniz $n$-algebra $L$ is said to be Cartan subalgebra if
\begin{itemize}
  \item[a)] $C$ is 1-nilpotent;
 \item[b)] $C=N_1(C)$.
 \end{itemize}
 \end{defn}

The importance of considering 1-normalizer in the definition of
Cartan subalgebras was shown in \cite{Omirov}.

\begin{pr}[\cite{cartan}] Let $C$ be a nilpotent subalgebra of a Leibniz $n$-algebra $L$.
Then $C$ is a Cartan subalgebra if and only if it coincides with $L_0$
in the Fitting decomposition of the algebra $L$ with respect to $R(C)$.
\end{pr}

Similarly as in \cite{Kasymov}, if $L$ is a direct sum of Leibniz
$n$-algebras $L_i,\, 1\leq i \leq k$, and $C_i$ are Cartan
subalgebras of $L_i$, then $C= \oplus_{i=1}^k C_i$ is a Cartan subalgebra of
$L$ and any Cartan subalgebra of $L$ has the same form.

The following result
concerning the regular elements of a Leibniz $n$-algebra was established in \cite{cartan}:

\begin{thm}[\cite{cartan}] Let $L$ be a Leibniz $n$-algebra over an infinite
field and let $x$ be a regular element for $L$. Then the Fitting
null-component $L_0$ with respect to the operator $R(x)$ is a
1-nilpotent subalgebra of $L$.
\end{thm}

In Leibniz algebras and Lie $n$-algebras the corresponding theorem states
that $L_0$ is a Cartan subalgebra. However, in \cite{cartan} we give
an example of a Leibniz $n$-algebra in which this result is not
true. Now we establish this result under some restrictions.

\begin{pr}Let $L$ be a Leibniz $n$-algebra over an infinite
field and let $x=(x_2,\dots,x_n)\in L^{\times (n-1)}$ be a regular
element for $L$ such that $x_2,\dots,x_n \in
L_0(R(x_2,\dots,x_n))$. Then the Fitting null-component $L_0$
with respect to operator $R(x)$ is a Cartan subalgebra of $L$.
\end{pr}

\begin{proof} Due to the previous theorem, we need to prove $N_1(L_0)=L_0$.
Let $y\in N_1(L_0)$. Then $[y,x_2,\dots,x_n]\in
[y,L_0,\dots,L_0]\subseteq L_0$. Hence $y\in L_0$. Therefore,
$N(L_0)\subseteq L_0$ and since $L_0$ is a subalgebra $N_1(L_0)\supseteq
L_0$. Thus, $L_0=N_1(L_0)$ and $L_0$ is a Cartan subalgebra of $L$.
\end{proof}

Now let us construct a Leibniz $n$-algebra $L$ such that the quotient $n$-algebra $L/I$ is a
simple Lie $n$-algebra, where
\[ I =  \mbox{ideal} \  \langle [x_1,\dots,x_i, \dots
, x_j, \dots , x_n] \ | \ \exists \, i, j: x_i=x_j \rangle\] is an
ideal of $L$.

\begin{exam}\label{examp_simple}
Let $\{e_1,\dots, e_{n+1},x_1,\dots, x_m\}$ be a basis of $L$.

Consider an algebra with the following multiplication:
\begin{align*}[e_{1}, \dots, e_{i-1},e_{i+1},\dots,
e_{n+1}] \ = & \ e_{i} \\
[x_k,e_j,\dots, e_j] \ = & \  \alpha_{kj} x_k \, ,
\end{align*}
where $ 1 \leq i,j\leq n+1,\, 1\leq k \leq
m, \ |\alpha_{k1}|^2+\dots+|\alpha_{k \, n+1}|^2 \neq 0$ for all $k$,
and the multiplication is skew symmetric in all the variables on $\langle
e_1,\dots,e_{n+1}\rangle$.

Then this algebra is a Leibniz $n$-algebra and $I=\langle x_1,
\dots, x_m\rangle $.
\end{exam}

Note that  $L/I$ is a simple Lie $n$-algebra and by \cite[Theorem 2.2]{Bai1}
we have that $F(L/I)=0$. Hence $F(L) \subseteq I$.

\begin{pr} In Example \ref{examp_simple}, $F(L)=0$.
\end{pr}
\begin{proof}
Consider the subspaces \[L_k=\langle e_1,\dots, e_{n+1}, x_1,\dots,
x_{k-1}, x_{k+1}, \dots, x_m \rangle ,\, 1\leq k \leq m.\] From
the multiplication table we get that they are subalgebras. Since
the dimension of these subalgebras is $n+m=\dim L -1$, they are
maximal subalgebras.

Hence, $\displaystyle F(L) \subseteq \bigcap_{k=1}^m L_k =\langle
e_1,\dots, e_{n+1}\rangle$. But $F(L) \subseteq I=\langle x_1,
\dots, x_m \rangle$. Thus $F(L)=0$.
\end{proof}

Below, we present a more general construction.

Let us consider an arbitrary Lie $n$-algebra with the basis
$e_1,\dots e_{n+1}$ and the conditions
\[[e_i,f_2,\dots, f_n]\in \langle e_1,\dots, e_{i-1},
e_{i+1},\dots, e_{n+1}\rangle,\] for all $f_2,\dots, f_n \in
\{e_1,\dots, e_{n+1}\}, \, 1\leq i \leq n+1$.

One of the Lie $n$-algebras with these conditions is a simple
Lie $n$-algebra. Complement this algebra with independent vectors
$x_1,\dots,x_m$ with the following multiplication
\[[x_k,e_p,\dots, e_p]=\alpha_{kp}^1x_1+\alpha_{kp}^2 x_2 +\dots +
\alpha_{kp}^m x_m\] for all $1\leq k \leq m, 1\leq p \leq n+1$.
Checking identity \eqref{FI} we will find restrictions on the
coefficients $\alpha_{ij}^k:$
\[\sum_{i=1}^m \alpha_{kp}^i \alpha_{iq}^j =\sum_{i=1}^m
\alpha_{kq}^i \alpha_{ip}^j\] for all $1\leq k,j\leq m,\, 1\leq p,
q \leq n+1$.

Hence, the satisfaction of the above condition guaranties that the
supplemented algebra is a Leibniz $n$-algebra.

Particularly, in this way, one can supplement simple Lie
$n$-algebras till Leibniz $n$-algebras.

On the ground of Example \ref{examp_simple} we give the following

\begin{exam} \label{cartan_examp} Let $L_s \ (1\leq s \leq n+1)$ be a
Leibniz $n$-algebra with the basis $\langle e_1,e_2,\dots,
e_{n+1},x_1,\dots,x_m\rangle$ and the following multiplication:
\begin{align*}
[e_1,\dots,e_{p-1},e_{p+1},\dots,e_{n+1}] \ = & \ e_p\, ,  & 1\leq p \leq n+1,\\
[x_k,e_k,e_k,\dots,e_k] \ = &  \ x_k \,,  & 1\leq k \leq s,\\
[x_{s+i},e_s,e_s,\dots,e_s] \ = &  \ x_{s+i}\, , &  1\leq i \leq m-s \, ,
\end{align*}
where the multiplication is skew symmetric in all the variables on
$\langle e_1,e_2,\dots, e_{n+1}\rangle$.

Then \[\begin{array}{rl}
  H_1= & \langle e_1,e_2,\dots,e_s, e_{s+1},\dots,e_{n-1}\rangle \\
  H_2= & \langle e_1,e_2,\dots,e_s, e_{s+2},\dots,e_{n-1},e_n \rangle \\
  H_3= & \langle e_1,e_2,\dots,e_s, e_{s+3},\dots,e_{n-1},e_{n+1} \rangle \\
\end{array}\]
are $n-1$ dimensional Cartan subalgebras.

The subalgebras
\begin{align*}
N_1  \ = & \  \langle x_1,e_2,e_3,\dots,e_n \rangle \\
N_2  \ = & \  \langle e_1,x_2,e_3,\dots,e_n \rangle\\
 \vdots & \\
N_{s-1}  \ = & \  \langle e_1,\dots,e_{s-2},x_{s-1},e_s,\dots,e_n \rangle
\end{align*}
are $n$ dimensional Cartan subalgebras.

The subalgebras
\begin{align*}
M_1  \ = & \  \langle e_1,e_2,\dots,e_{s-1},e_{s+1},e_{s+2},\dots,e_n,x_s,x_{s+1},\dots, x_m \rangle \\
M_2   \ = & \  \langle e_1,e_2,\dots,e_{s-1},e_{s+2},e_{s+3},\dots,e_{n+1},x_s,x_{s+1},\dots, x_m \rangle
\end{align*}
are $m+n-s$ dimensional Cartan subalgebras.
\begin{align*}
C_1  \ = & \  \langle x_1,e_2,\dots,e_{s-1},e_{s+1},\dots,e_{n+1},x_s,x_{s+1},\dots, x_m \rangle \\
C_2  \ = & \  \langle e_1,x_2,\dots,e_{s-1},e_{s+1},\dots,e_{n+1},x_s,x_{s+1},\dots, x_m \rangle \\
 \vdots  & \\
C_{s-1} \ = & \  \langle e_1,\dots,e_{s-2},x_{s-1},e_{s+1},\dots,e_{n+1},x_s,x_{s+1},\dots, x_m \rangle
\end{align*}
are $m+n-s+1$ dimensional Cartan subalgebras.
\end{exam}

In the considered Leibniz $n$-algebra we found Cartan subalgebras of dimensions
$n-1, n , n+m-s$ and $n+m-s+1$. Hence, in general, Cartan subalgebras of a given
Leibniz $n$-algebra are not conjugated.

Here we give a theorem that establishes the conjugacy  of Cartan
subalgebras under some restrictions on the Leibniz $n$-algebra.

\begin{thm} \label{condition} Let $L$ be a finite dimensional Leibniz $n$-algebra and $H$
be a Cartan subalgebra of $L$. Suppose that
\begin{itemize}
  \item[(i)]  the multiplication is  skew symmetric in the first two variables; and that
   \item[(ii)] for any element $h=(h_1,\dots,h_{n-1})\in H^{\times (n-1)}$, we
have $h_i\in \Ker R(h)$ for all $1\leq i \leq n-1$.
\end{itemize}
Then there is a regular element $h\in H^{\times (n-1)}$ such that
$H=L_0(R(h))$.
\end{thm}

\begin{proof} Suppose that $H$ is a Cartan subalgebra of a Leibniz
$n$-algebra and $L=L_0\oplus L_{\alpha_1}\oplus\dots\oplus
L_{\alpha_s}$ is the decomposition of $L$ into a direct sum of
root subspaces with respect to $H$ and $\Delta =\{ \alpha_1,\dots,
\alpha_s\}$ is the set of non-zero roots of $H$ in $L$. Then the
functions $\alpha_i$ are multilineal and, in particular,
polynomial.  Since $H^{\times (n-1)}$ is an irreducible variety,
it follows that $\alpha_1\alpha_2\cdots \alpha_s$ is also a
non-zero polynomial function from $H^{n-1}$ to the ground field of the
Leibniz $n$-algebra. Hence, $\alpha_1(h^0)\alpha_2(h^0)\cdots
\alpha_s(h^0)\neq 0$ for some $h^0=(h_0,h^0_1,\dots,h^0_{n-2})$.
This means that the characteristic roots of the restriction
$\overline{R}(h^0)$ of the endomorphism $R(h^0)$ to
$L_1\big(R(h^0)\big)=\sum_{\alpha\in \Delta} L_{\alpha}$ are all nonzero,
and hence $\overline{R}(h^0)$ is a non-degenerate operator.

The proof of the theorem is based on the proof of conjugacy of Cartan
subalgebras in Lie $n$-algebras given by Kasymov \cite{Kasymov}.
Similarly, we define a polynomial function $P$ on $L$ by
\[P(x)=\exp R(x_1,h^0_1,\dots, h^0_{n-2})\cdots \exp R(x_s,h^0_1,\dots,
h^0_{n-2})(h),\] where $x=h+x_1+\dots+x_s, h\in H=L_0\big(R(h^0)\big), \ x_i
\in L_{\alpha_i}$.

Notice that, if a right multiplication operator $R(x)$ is
nilpotent, then $\exp R(x)$ is an inner automorphism of the
Leibniz $n$-algebra $L$. Automorphisms of this kind generate a
certain subgroup $G_0$ in the group $G=\Aut L$. Elements of $G_0$
are called \emph{special (invariant) automorphisms}.

Using the skew symmetrical property of the multiplication in the first two
variables, we establish that the differential $d_{h^0}P$ of $P$ at
a point $h^0$ is an epimorphism. Hence, by facts from algebraic
geometry in \cite{Kasymov}, this polynomial function $P$ is
dominating, i.e. for any non-zero polynomial function $f$ on $L$ there
exists a non-zero polynomial function $g$ on $L$ such that every
$y\in L$ with $g(y)\neq 0$ is represented as $y=P(x)$, where
$f(x)\neq 0$.

Assuming that for any regular element $h=(h_1,\dots,h_{n-1})\in
H^{\times (n-1)}$ we have $h_i\in \Ker R(h)$ for all $1\leq i \leq
n-1$, then
\[P(h_i)=\left(\prod_{j=1}^s \exp
R(x_j,h_1,\dots,h_i,\dots, h_{n-2})\right) (h_i)=h_i.\] Hence, we
can use similar induction as in \cite{Kasymov} to prove the
existence of a regular element $h\in H^{\times (n-1)}$ such that
$H=L_0\big(R(h)\big)$.
\end{proof}

Under the conditions of Theorem \ref{condition} the following
theorem can be proved similarly as in the case of Lie
$n$-algebras \cite{Kasymov}.

\begin{thm} Let $L$ be a finite-dimensional Leibniz $n$-algebra
which satisfies the conditions (i)-(ii) of Theorem
\ref{condition}. If $H$ and $K$ are Cartan subalgebras of $L$,
then there exists a special automorphism $\delta$ of $L$ such that
$H=\delta(K)$.
\end{thm}

\section*{Acknowledgments}
 The  second author was supported by MICINN grant MTM 2009-14464-C02 (European
FEDER support included) and by Xunta de Galicia grant Incite 09207215 PR. The third named author would like to
acknowledge the hospitality of the University of Santiago de
Compostela (Spain). He was supported by Grant NATO-Reintegration
ref. CBP.EAP.RIG. 983169.

The last named author would like to acknowledge ACTP OEA-AC-84 for
a given support.

\address{\small \rm  Felipe Gago: Departamento de \'Algebra,  Universidad de Santiago de Compostela, 15782
Santiago de Compostela, Spain}\\ \email{felipe.gago@usc.es}

\address{\small \rm  Manuel Ladra: Departamento de \'Algebra,  Universidad de Santiago de Compostela, 15782
Santiago de Compostela, Spain}\\ \email{manuel.ladra@usc.es}

\address{\small \rm  Bakhrom Omirov: Institute of Mathematics and Information Technologies,
29, Dormon Yoli, 100125 Tashkent,
Uzbekistan}\\ \email{omirovb@mail.ru}

\address{\small \rm  Rustam Turdibaev: Department of Mathematics, National University of Uzbekistan,
Vuzgorogok, 27, 100174 Tashkent, Uzbekistan}\\ \email{rustamtm@yahoo.com}

\end{document}